\documentclass[12pt]{amsart}

\usepackage{a4wide}
\usepackage{amsfonts}
\usepackage{amssymb}

\numberwithin{equation}{section}

\theoremstyle{plain}

\newtheorem{theorem}[subsection]{Theorem}
\newtheorem{proposition}[subsection]{Proposition}

\newtheorem{lemma}[subsection]{Lemma}

\theoremstyle{definition}

\newtheorem{question}[subsection]{Question}

\providecommand{\F}{\mathop{\mathbb{F}}\nolimits}
\providecommand{\E}{\mathop{\mathbb{E}}\nolimits}

\providecommand{\Ham}{\mathop{\rm Ham}\nolimits}

\newcommand{\wh}{\widehat}
\renewcommand{\P}{\mathop{\mathbb{P}}\nolimits}

\begin{document}

\title{Green's sumset problem at density one half}

\author{Tom Sanders}
\address{Department of Pure Mathematics and Mathematical Statistics\\
University of Cambridge\\
Wilberforce Road\\
Cambridge CB3 0WA\\
England } \email{t.sanders@dpmms.cam.ac.uk}

\begin{abstract}
We investigate the size of subspaces in sumsets and show two main results.  First, if $A \subset \F_2^n$ has density at least $1/2 - o(n^{-1/2})$ then $A+A$ contains a subspace of co-dimension $1$.  Secondly, if $A \subset \F_2^n$ has density at least $1/2-o(1)$ then $A+A$ contains a subspace of co-dimension $o(n)$.\end{abstract}

\maketitle

\section{Introduction}

In the paper \cite{JBAA} Bourgain first addressed the question of showing that if $A \subset \{1,\dots,N\}$ has positive relative density then $A+A$ contains a very long arithmetic progression. Since his work the problem has received considerable attention, and to help understand it better Green \cite{BJGFFM} introduced a model version which has turned out to be interesting in its own right. It is this question with which we shall concern ourselves in this note.

Suppose, as we shall throughout, that $G:=\F_2^n$ and let $\P_G$ denote the normalized counting measure on $G$. We are interested in what size of subspace one can guarantee that $A+A$ contains, where $A$ is a subset of $G$ of density $\alpha:=\P_G(A)$. 

It turns out that there are various ranges of the density in which we see quite different phenomena. To begin note that if $\alpha>1/2$ then the inclusion-exclusion principle tells us that $\P_G(A \cap (x+A))>0$ for all $x \in G$ and so we have that $A+A=G$; we write this as follows.
\begin{proposition}\label{prop.triv}
Suppose that $A\subset G:=\F_2^n$ has density $\alpha > 1/2$. Then $A+A$ contains a subspace of co-dimension $0$.
\end{proposition}
Once the density dips below $1/2$ things begin to change. In this regime $A$ may be contained in a subspace of co-dimension $1$, and so $A+A$ can, at best, be guaranteed to contain a subspace of co-dimension $1$. To start with this is best possible:
\begin{theorem}\label{thm.epssmall}
Suppose that $A\subset G:=\F_2^n$ has density $\alpha > 1/2-\epsilon$ where $\epsilon \in (0,1/2^9\sqrt{n}]$. Then $A+A$ contains a subspace of co-dimension $1$.
\end{theorem}
Once $\epsilon \sim 1/\sqrt{n}$, however, a different sort of behavior manifests. The worst known such is exhibited by the so called Niveau set construction of Ruzsa \cite{IZREC} and for us this yields the following theorem.
\begin{theorem}[Green-Ruzsa]\label{thm.construction}
For all $\epsilon \in (2^2/\sqrt{n},1/2]$, there is a set $A=A(\epsilon) \subset G:=\F_2^n$ of density $\alpha > 1/2 - \epsilon$ such that any subspace contained in $A+A$ has co-dimension $\Omega(\epsilon \sqrt{n})$.
\end{theorem}
We are not able to prove a matching upper bound, although we are able to establish the following weak complement.
\begin{theorem}\label{thm.new}
Suppose that $A\subset G:=\F_2^n$ has density $\alpha > 1/2 - \epsilon$ where $\epsilon \in (0,1/2]$. Then $A+A$ contains a subspace of co-dimension $O(\frac{n}{ \log \epsilon^{-1}})$.
\end{theorem}
All previous work on showing that sumsets contain large subspaces has concentrated on the case of positive density (rather than density close to $1/2$) and has consequently produced weaker results.
It turns out the argument used for the previous theorem also yields a new result in this case.
\begin{theorem}\label{thm.new2}
Suppose that $A\subset G:=\F_2^n$ has density $\alpha > 0$. Then $A+A$ contains a subspace of dimension $\Omega(\alpha n)$.
\end{theorem}
This improves upon \cite[Theorem 9.3]{BJGFFM} where a lower bound on the dimension of the form $\Omega(\alpha^2 n)$ was given.  Recently Croot and Sisask \cite{ESCOS} also established an improvement of roughly the strength of Theorem \ref{thm.new2} but by different arguments in a much more general setting.

Finally, Fourier analysis is notoriously weak when dealing with thin sets and, indeed, if $\alpha = o(\log n/ n)$ then it turns out that an elementary counting argument of Croot, Ruzsa and Schoen \cite{ESCIZRTS} supersedes Theorem \ref{thm.new2}.

As indicated we make use of the Fourier transform the group $G:=\F_2^n$.  In particular we denote the dual group by $\wh{G}$ and define the transform to be the map taking $f \in L^1(G)$ to
\begin{equation*}
\wh{f}(\gamma):=\E_{x \in G}{f(x)\overline{\gamma(x)}}.
\end{equation*}
We use basic results from Fourier analysis without comment and the reader interested in details may wish to consult Tao and Vu \cite{TCTVHV}.

he note now splits into three further sections in which Theorem \ref{thm.epssmall}, Theorem \ref{thm.construction} and Theorem \ref{thm.new} (and \ref{thm.new2}) are proved, followed by some concluding remarks in the final section including a discussion of the link with the integer version of the problem.

\section{Proof of Theorem \ref{thm.epssmall}}

The argument involves two main tools. The first is Pl{\"u}nnecke's inequality, \cite{HP}, which we record now. One of the key ideas in our work is to make use of it in the region when $K \sim 1$.
\begin{theorem}[{\cite[Corollary 6.28]{TCTVHV}}]
Suppose that $A,B \subset G:=\F_2^n$ are such that $\P_G(A+B) \leq K\P_G(A)$. Then for any positive integer $k$ there is a set $X \subset A$ with $\P_G(X+kB) \leq K^k\P_G(X)$.
\end{theorem}
The second tool is measure concentration on the cube for which we shall follow McDiarmid \cite{CM}.  The idea of using measure concentration was introduced to difference set problems by Wolf in \cite{JWSPDS} and that is the inspiration for our application.

For any natural number $n$ we write $Q_n$ for the cube $\{0,1\}^n$.  The Hamming metric on $Q_n$ is defined in the usual way:
\begin{equation*}
d(x,y):=|\{i:x_i \neq y_i\}| \textrm{ for all } x,y \in Q_n.
\end{equation*}
For any set $A \subset Q_n$ and $r\geq 0$ we write
\begin{equation*}
\Ham_r(A):=\{x \in Q_n: d(x,y) \leq r \textrm{ for some } y \in A\},
\end{equation*}
that is the set of points of at most distance $r$ from $A$.  Measure concentration provides a lower bound for the density of this set.
\begin{theorem}[{\cite[Proposition 7.7]{CM}}]\label{thm.concentrateG}
Suppose that $A$ is a non-empty subset of $Q_n$.  Then for any $r\geq 0$ we have
\begin{equation*}
\P_{Q_n}(\Ham_r(A)) \geq 1-\frac{\exp(-r^2/2n)}{\P_{Q_n}(A)}.
\end{equation*}
\end{theorem}
The formal similarity of $\F_2^n$ and $Q_n$ immediately tells us how we shall make use of this result.  Suppose that $E=\{e_1,\dots,e_n\}$ is a basis of $G:=\F_2^n$, so that the map 
\begin{equation*}
\phi_E: Q_{n} \rightarrow G; x \mapsto x_1.e_1+\dots+x_n.e_n
\end{equation*}
is a bijection. Sets in $G$ inherit certain growth properties from those in $Q_n$ as follows. Writing $F:=E \cup \{0_G\}$ we have that
\begin{equation}\label{eqn.ku}
A+rF \supset \phi_E(\Ham_r(\phi_E^{-1}( A))).
\end{equation}
To see this note that if $z$ is a member of the right hand side then there is some $y \in A$ such that
\begin{equation*}
d(\phi_E^{-1}(z),\phi_E^{-1}(y)) \leq r.
\end{equation*}
Note that this is perfectly well defined since $\phi_E^{-1}$ is a bijection.  Now, let $x \in \{0,1\}^n$ be such that
\begin{equation*}
x_i = \begin{cases} 1 & \textrm{ if } \phi_E^{-1}(z)_i \neq \phi_E^{-1}(y)_i\\ 0 & \textrm{ otherwise,}\end{cases}
\end{equation*}
so that
\begin{equation*}
\phi_E^{-1}(z) = \phi_E^{-1}(y + \phi_E(x)).
\end{equation*}
On the other hand the number of $i$ such that $x_i \neq 0$ is at most $r$ and $0_G \in F$, whence $z \in y+rF$.  (\ref{eqn.ku}) then follows.  In light of this we have the following consequence of Theorem \ref{thm.concentrateG}.
\begin{proposition}\label{prop.concentrate}
Suppose that $G=\F_2^n$ and $E$ is a basis of $G$, $A$ is a non-empty subset of $G$ and $F:=E \cup  \{0_G\}$.  Then for any $r\geq 0$ we have
\begin{equation*}
\P_{G}(A+rF) \geq 1-\frac{\exp(-r^2/2n)}{\P_{G}( A)}.
\end{equation*}
\end{proposition}

The heart of the argument is the following asymmetric version of Theorem \ref{thm.epssmall}.
\begin{proposition}\label{prop.epssmall}
Suppose that $A\subset G:=\F_2^n$ has density $\alpha > 1/2-\epsilon$ where $\epsilon \in (0,1/2^9\sqrt{n}]$. Then $A+A$ contains a coset of a subspace of co-dimension $1$.
\end{proposition}
\begin{proof}
In view of Proposition \ref{prop.triv} we may certainly assume that $\P_G(A) \leq 1/2$. Put $S:=(A+A)^c$ and note that $(S+A) \cap A = \emptyset$, whence $\P_G(A+S) \leq 1-\P_G(A)$. Since $\P_G(A)=\alpha > 1/2-\epsilon$, it follows that
\begin{equation*}
\P_G(A+S) < \frac{1+2\epsilon}{1-2\epsilon}\P_G(A) \leq \exp(6\epsilon)\P_G(A),
\end{equation*}
since $\epsilon \leq 1/4$.  Suppose that $E$ is a basis of $G$ such that there is some $s\in S$ with $s+E \subset S$.  Since $\P_G$ is translation invariant we have that $F:=E \cup \{0_G\}$ has
\begin{equation*}
\P_G(A+F) = \P_G(A+ (s+F)) \leq \P_G(A+S) \leq \exp(6\epsilon)\P_G(A).
\end{equation*}
Now, let $k=\lceil \sqrt{n}\rceil$ and define a sequence of sets $X_0,X_1,\dots$ using Pl{\"u}nnecke's inequality: let $\emptyset \neq X_0 \subset A$ be such that $\P_G(X_0+kF) < \exp(6\epsilon k)\P_G(X_0)$, and $\emptyset \neq X_{r+1} \subset X_r$ be such that
\begin{equation}\label{eqn.growth}
\P_G(X_{r+1}+2^{r+1}kF) \leq \exp(6\epsilon.2^{r+1}k)\P_G(X_{r+1}).
\end{equation}
Since $X_r \neq \emptyset$ and $(X_r)_r$ is nested, we see that the sequence $(\sqrt{2\log 8\P_G(X_r)^{-1}})_r$ is increasing and bounded above by $O(\sqrt{n})$. It follows that there is some $r$ such that $2^r \geq \sqrt{2\log 8\P_G(X_r)^{-1}}$; let $r'$ be the minimal such. In this case
\begin{equation*}
2^{r'-1} <\sqrt{2\log 8\P_G(X_{r'-1})^{-1}} \leq \sqrt{2\log 8\P_G(X_{r'})^{-1}}
\end{equation*}
by nesting of $(X_r)_r$, whence
\begin{equation}\label{eqn.rest}
\sqrt{2\log 8\P_G(X_{r'})^{-1}} \leq 2^{r'} \leq 2.\sqrt{2\log 8\P_G(X_{r'})^{-1}}\leq 2^5 \log \frac{3\P_G(X_{r'})^{-1}}{4}.
\end{equation}
But by Proposition \ref{prop.concentrate} we have that
\begin{equation*}
\P_G(X_{r'} + 2^{r'}kF) + \frac{\exp(-2^{2r'}k^2/2n)}{\P_G(X_{r'})} \geq 1.
\end{equation*}
Now by (\ref{eqn.growth}) and the upper bound in (\ref{eqn.rest}) we get that
\begin{equation*}
\P_G(X_{r'}+2^{r'}kE) \leq \exp(6\epsilon.2^{r'}k)\P_G(X_{r'}) \leq \frac{3}{4}
\end{equation*}
since $k \leq 2\sqrt{n} \leq 2^{-8}\epsilon^{-1}$ by assumption on $\epsilon$; whence
\begin{equation*}
\frac{\exp(-2^{2r'}k^2/2n)}{\P_G(X_{r'})}  \geq \frac{1}{4}.
\end{equation*}
On the other hand this can be bounded above using the lower bound in (\ref{eqn.rest}) and the fact that $k^2 \geq n$:
\begin{equation*}
\frac{\exp(-2^{2r'}k^2/2n)}{\P_G(X_{r'})}  \leq \frac{(8\P_G(X_{r'})^{-1})^{-k^2/n}}{\P_G(X_{r'})} \leq \frac{1}{8}.
\end{equation*}
This contradiction means that for all $s \in S$, the set $S-s$ contains at most $n-1$ linearly independent vectors.  Thus there is an element $s \in S$ and a subspace $H$ of co-dimension $1$ in $G$ such that $S \subset s+H$. Since $S=(A+A)^c$ it follows that $A+A \supset (s+H)^c$; $(s+H)^c$ is simply the coset of $H$ not equal to $s+H$, whence we are done.
\end{proof}
We now use the above result in a straightforward manner.
\begin{proof}[Proof of Theorem \ref{thm.epssmall}]
Let $U$ be the smallest subspace of $G$ such that $A+A \supset U^c$. Note that such a space exists since $A+A \supset \emptyset = G^c$.

By averaging there is a coset $x+U$ on which $A$ has relative density at least $\alpha$: write $A':=(A \cap (x+U) -x)$ and note that $A+A \supset A'+A'$ and $\P_U(A') \geq \alpha$. By Proposition \ref{prop.epssmall} there is a subspace $U' \leq U$ of relative co-dimension at most $1$ and some $u \in U$ such that $u+U'\subset A'+A' \subset A+A$. We have three cases:
\begin{enumerate}
\item $U'=U$: then 
\begin{equation*}
G= U \cup U^c = U' \cup U^c \subset A+A;
\end{equation*}
\item $U' \neq U$ and $u+U' \neq U'$: then $A+A \supset U'^c$ and $\dim U' < \dim U$ contradicting the minimality of $U$;
\item $U' \neq U$ and $u+U' = U'$: then let $\pi:G \rightarrow U$ be a projection which is the identity when restricted to $U$ and note that $\pi^{-1}(U')$ is a subspace of $G$ of co-dimension $1$ with
\begin{eqnarray*}
\pi^{-1}(U') & =& ( \pi^{-1}(U')  \cap U) \cup (\pi^{-1}(U') \cap U^c)\\ & \subset & U' \cup U^c \subset A+A.
\end{eqnarray*}
\end{enumerate}
The result follows.
\end{proof}

\section{Proof of Theorem \ref{thm.construction}}

The argument here is a very slight adaptation of \cite[Theorem 9.4]{BJGFFM}. Green established this by reformulating Ruzsa's construction from \cite{IZRAA} in the model setting where many of the details simplify.
\begin{proof}[Proof of Theorem \ref{thm.construction}]
Let 
\begin{equation*}
A:=\{x \in \F_2^n: x \textrm{ has at most } n/2 - \eta\sqrt{2\pi n}/2 \textrm{ ones.}\}.
\end{equation*}
Let $X$ be the random variable which takes $x \in G$ to the number of $1$s in $x$. $\P_G$ is the uniform distribution on $G$, and $X$ is the sum of $n$ independent identically distributed Bernoulli random variables with parameter $p=1/2$. 

The mean of each individual Bernoulli trial is $1/2$ and the variance $1/4$ so that the mean of $X$ is $n/2$ and the variance is $n/4$.  It follows from the Berry-Esseen inequality (see, e.g. \cite[p374]{ANS}) that
\begin{equation*}
\sup_{x \in G}{|\P_G(A) - \Phi(- \eta \sqrt{2\pi})|} \leq \frac{3.2}{\sqrt{n}}.
\end{equation*}
Thus
\begin{equation*}
\P_G(A) \geq \Phi(- \eta \sqrt{2\pi}) - \frac{3.2}{\sqrt{n}}.
\end{equation*}
On the other hand
\begin{equation*}
 \Phi(2\lfloor \eta \sqrt{2\pi n}/2\rfloor/\sqrt{n}) =  \frac{1}{2} - \frac{1}{\sqrt{2\pi}}\int_0^{\eta \sqrt{2\pi}}{\exp(-x^2/2)dx} \geq   \frac{1}{2} -\eta.
\end{equation*}
It follows that
\begin{equation*}
\P_G(A) \geq \frac{1}{2} - \eta -\frac{3.2}{\sqrt{n}}.
\end{equation*}
It follows that for $\epsilon\sqrt{n}\geq 2^2$ we may pick $\eta=0.8\epsilon$ and get that $\P_G(A) > 1/2 - \epsilon$.

Now we shall show that if $V \leq G$ has co-dimension at most $d:=\lfloor \eta \sqrt{2\pi n} \rfloor$ then $V$ contains a vector with at most $\lfloor \eta \sqrt{2\pi n} \rfloor$ zeros in the standard basis. Since any $x\in A+A$ has at least $\eta\sqrt{2\pi n}$ zeros in the standard basis we shall be done. Such a $V$ can be written as
\begin{equation*}
V=\{\lambda_1v_1+\dots + \lambda_{n-d}v_{n-d}: \lambda_i \in \F_2\},
\end{equation*}
where the $v_i$s are linearly independent. The $v_i$s may be written in the standard basis as
\begin{equation*}
v_i=\epsilon_i^{(1)}e_1+ \dots + \epsilon_i^{(n)}e_n
\end{equation*}
where $(e_i)_i$ is the standard basis. The column rank of the matrix $(\epsilon_i^{(j)})_{ij}$ is $n-d$ hence so is its row rank. Without loss of generality we may suppose that the first $n-d$ rows $(\epsilon_1^{(j)},\dots,\epsilon_{n-d}^{(j)})$, $j=1,\dots,n-d$ are linearly independent. It follows that we can solve the $n-d$ equations
\begin{equation*}
\lambda_1\epsilon_1^{(j)} + \dots + \lambda_{n-d}\epsilon_{n-d}^{(j)} =1
\end{equation*}
for the $\lambda_i$ giving a vector in $V$ with no more than $d$ zeros. The result follows.
\end{proof}

\section{Proof of Theorem \ref{thm.new}}

We shall prove the following stronger theorem from which both Theorems \ref{thm.new} and \ref{thm.new2} follow.
\begin{theorem}\label{thm.underlying}
Suppose that $A\subset G:=\F_2^n$ has density $\alpha \leq 1/2$. Then $A+A$ contains a subspace of co-dimension
\begin{equation*}
\left\lceil n/\log_2\frac{2-2\alpha}{1-2\alpha}\right\rceil.
\end{equation*}
\end{theorem}
Before proving this we establish our two consequences.
\begin{proof}[Proof of Theorem \ref{thm.new}]
Apply Theorem \ref{thm.underlying} with $\alpha =1/2 - \epsilon$.
\end{proof}
\begin{proof}[Proof of Theorem \ref{thm.new2}]
It is easy to see that
\begin{equation*}
\left\lceil n/\log_2\frac{2-2\alpha}{1-2\alpha}\right\rceil \leq n(1-\alpha/\log 2)+O(1).
\end{equation*}
Theorem \ref{thm.underlying} then tells us that $A+A$ contains a subspace of dimension at least $\alpha n/\log 2 - O(1)$.
\end{proof}
The proof is inspired by the standard iterative method of Roth introduced in \cite{KFRPre} and the more commonly cited \cite{KFR}, which was adapted to finite fields by Meshulam in \cite{RM}. The following lemma is the driver.
\begin{lemma}[Iteration lemma]\label{lem.iteration}
Suppose that $A \subset G:=\F_2^n$ has density $\alpha>0$ and $V \leq G$. Then there is a subspace $V' \leq V$ of relative co-dimension $1$ such that 
\begin{equation*}
\P_{V'}(V' \setminus (A+A)) \leq \frac{1-2\alpha}{1-\alpha}\P_V(V \setminus (A+A)).
\end{equation*}
\end{lemma}
\begin{proof}
For each $W \in G/V$ fix some $x_W \in W$ and let $A_W$ be $A\cap W -x_W$ considered as a subset of $V$. Since $-2x_W=0$ we have that $A_W+A_W = A \cap W + A \cap W$, whence
\begin{equation*}
(A+A) \cap V = \bigcup_{W \in G/V}{(A_W + A_W)};
\end{equation*}
write $S:=V \setminus (A+A)$. In view of the definition of $S$ and our above observation, we get from Plancherel's theorem that
\begin{equation}\label{eqn.base}
0=\langle \sum_{W \in G/V}{1_{A_W} \ast 1_{A_W}},1_S \rangle_{L^2(V)} = \sum_{\gamma \in \wh{V}}{\sum_{W \in G/V}{|\wh{1_{A_W}}(\gamma)|^2}\wh{1_S}(\gamma)}.
\end{equation}
Partition $\wh{V}$ into two sets $\mathcal{N}:=\{\gamma \in \wh{V}: \wh{1_S}(\gamma)<0\}$ and $\mathcal{P}:=\{\gamma \in \wh{V}: \wh{1_S}(\gamma)\geq 0\}$. Since $0_{\wh{V}} \in \mathcal{P}$ we have
\begin{equation*}
\sum_{\gamma \in \mathcal{P}}{\sum_{W \in G/V}{|\wh{1_{A_W}}(\gamma)|^2}\wh{1_S}(\gamma)} \geq \P_V(S)\sum_{W \in G/V}{\P_V(A_W)^2}.
\end{equation*}
On the other hand, from equation (\ref{eqn.base}) and the definition of $\mathcal{N}$, we get
\begin{eqnarray*}
\sum_{\gamma \in \mathcal{N}}{\sum_{W \in G/V}{|\wh{1_{A_W}}(\gamma)|^2}|\wh{1_S}(\gamma)|}&=&\sum_{\gamma \in \mathcal{N}}{\sum_{W \in G/V}{-|\wh{1_{A_W}}(\gamma)|^2}\wh{1_S}(\gamma)}\\  & = & \sum_{\gamma \in \mathcal{P}}{\sum_{W \in G/V}{|\wh{1_{A_W}}(\gamma)|^2}\wh{1_S}(\gamma)}\\ &\geq & \P_V(S)\sum_{W \in G/V}{\P_V(A_W)^2}.
\end{eqnarray*}
By the H{\"o}lder's inequality and then Parseval's theorem we have
\begin{eqnarray*}
\sum_{\gamma \in \mathcal{N}}{\sum_{W \in G/V}{|\wh{1_{A_W}}(\gamma)|^2}|\wh{1_S}(\gamma)|} &\leq & \sup_{\gamma \in \mathcal{N}}{|\wh{1_S}(\gamma)|}\sum_{\gamma\neq 0_{\wh{G}}}{\sum_{W \in G/V}{|\wh{1_{A_W}}(\gamma)|^2}}  \\& = & \sup_{\gamma \in \mathcal{N}}{|\wh{1_S}(\gamma)|}\sum_{W \in G/V}{(\P_V(A_W)-\P_V(A_W)^2)} .
\end{eqnarray*}
Combining this with the foregoing we get that
\begin{equation}
\label{eqn.lt}
 \P_V(S)\sum_{W \in G/V}{\P_V(A_W)^2} \leq \sup_{\gamma \in \mathcal{N}}{|\wh{1_S}(\gamma)|}\sum_{W \in G/V}{(\P_V(A_W)-\P_V(A_W)^2)}.
\end{equation}
Now, by the Cauchy-Schwarz inequality 
\begin{equation*}
\E_{W \in G/V}{\P_V(A_W)}. \sum_{W \in G/V}{\P_V(A_W)} \leq \sum_{W \in G/V}{\P_V(A_W)^2}.
\end{equation*}
However, $\E_{W \in G/V}{\P_V(A_W)}=\alpha$, whence
\begin{equation*}
\sum_{W \in G/V}{(\P_V(A_W)-\P_V(A_W)^2)} \leq (\alpha^{-1} -1)\sum_{W \in G/V}{\P_V(A_W)^2}.
\end{equation*}
Inserting this into (\ref{eqn.lt}) we get that
\begin{equation*}
\P_V(S)\sum_{W \in G/V}{\P_V(A_W)^2}  \leq \sup_{\gamma \in \mathcal{N}}{|\wh{1_S}(\gamma)|}(\alpha^{-1} -1)\sum_{W \in G/V}{\P_V(A_W)^2}.
\end{equation*}
Since $\alpha>0$ and $G/V$ is finite we may divide by the sum and conclude that 
\begin{equation*}
\sup_{\gamma \in \mathcal{N}}{|\wh{1_S}(\gamma)|} \geq \frac{\alpha}{1-\alpha}\P_V(S).
\end{equation*}
Let $\gamma \in \mathcal{N}$ be such that this supremum is attained and write $V':=\{\gamma\}^\perp$. $V'$ has relative co-dimension $1$ and, in view of this, we note that $V \setminus V' = x_0+V'$ for any $x_0 \in V \setminus V'$. Now
\begin{equation*}
\P_V(S\cap V') - \P_V(S \cap (x_0+V')) = \wh{1_S}(\gamma) \leq -\frac{\alpha}{1-\alpha}\P_V(S),
\end{equation*}
since $V'=\{\gamma\}^\perp$ and $\gamma\in \mathcal{N}$. Furthermore, since $V'$ and $x_0+V'$ partition $V$ we have that
\begin{equation*}
\P_V(S \cap V') + \P_V(S \cap (x_0+V')) =\P_V(S).
\end{equation*}
Adding these two expressions tells us that
\begin{equation*}
\P_{V'}(S \cap V') = 2\P_V(S \cap V') = \left(1-\frac{\alpha}{1-\alpha}\right)\P_V(S) = \frac{1-2\alpha}{1-\alpha} \P_V(S).
\end{equation*}
It remains only to note that $S \cap V' = V' \setminus (A+A)$ and the lemma is proved.
\end{proof}

\begin{proof}[Proof of Theorem \ref{thm.underlying}]
Let $\sigma \in (0,1)$ be a parameter to be optimized later. We use the iteration lemma to produce a sequence of subspaces $V_i$ such that
\begin{enumerate}
\item \label{enum.pt12}$V_i \leq V_{i-1}$ and the co-dimension of $V_i$ in $V_{i-1}$ is $1$;
\item \label{enum.pt42}$\P_{V_i}( V_i \setminus (A+A)) \leq \frac{1-2\alpha}{1-\alpha} \P_{V_{i-1}}( V_{i-1} \setminus (A+A))$.
\end{enumerate}
We set $V_0=G$, and apply the iteration lemma (Lemma \ref{lem.iteration}) repeatedly to get the sequence.  Now,  if $|V_i \setminus (A+A)| <1$ then $V_i \setminus (A+A)=\emptyset$, whence $A+A$ contains a subspace of co-dimension $i$ by (\ref{enum.pt12}). In view of (\ref{enum.pt12}) and (\ref{enum.pt42}) this certainly happens if
\begin{equation*}
|V_i|\P_{V_i}(V_i \setminus (A+A))  \leq 2^{n-i}\left(\frac{1-2\alpha}{1-\alpha}\right)^i<1;
\end{equation*}
taking $i$ minimal such that this inequality is satisfied yields the result.
\end{proof}

\section{Concluding remarks}

As noted in the introduction Theorem \ref{thm.epssmall} is best possible, however there is still a large gap between Theorem \ref{thm.construction} and Theorem \ref{thm.new}.  It has been suggested in \cite{BJGFFM} that the truth is closer to Theorem \ref{thm.construction} in the case when the density of the set is $\Omega(1)$.  One might take Theorem \ref{thm.epssmall} as some support of this conjecture (at least in the case of density $1/2-o(1)$).  An intermediate question might be the following.
\begin{question}
Suppose that $A \subset G:=\F_2^n$ has density $\alpha > 1/2-C/\sqrt{n}$.  Does $A+A$ contain a subspace of co-dimension $O_C(1)$?
\end{question}
Theorem \ref{thm.construction} tells us that this cannot be sublinear, and while showing that may be hard it could be that $O(C^2)$ is rather more accessible.

Theorem \ref{thm.new} (and Theorem \ref{thm.epssmall}) are the first results which provide a sensible upper bound on the co-dimension rather than lower bound on the dimension of the subspace found in $A+A$, and given the proof one might imagine an improvement to Theorem \ref{thm.new} would be possible. 

The following is a well-known theorem (see, for example, Metsch \cite{KM}).
\begin{theorem}\label{thm.proj}
Suppose that $G:=\F_2^n$ and $S \subset G\setminus \{0_G\}$ meets every subspace of dimension $d$. Then $|S| \geq 2^{n+1-d}-1$.
\end{theorem}
This result can be used in the proof of Theorem \ref{thm.underlying} above to tell us that once we have
\begin{equation*}
|V_i \setminus (A+A)| < 2^{n-i+1-d}
\end{equation*}
in the iteration then $A+A$ must contain a subspace of dimension $d$. This certainly happens when
\begin{equation*}
2^{1-d} > \left(\frac{1-2\alpha}{1-\alpha}\right)^i;
\end{equation*}
again taking $i$ minimal such that this inequality holds we get a saving of $1$ in the co-dimension.  This is, of course, not particularly impressive and, indeed, no na{\"\i}ve attack along these lines will work as it turns out that Theorem \ref{thm.proj} is also best possible.  (Again, see Metsch \cite{KM}.)  This may be contrasted with the following result of Alon.
\begin{theorem}[{\cite[Theorem 4.1]{NA}}]
Suppose that $G:=\F_2^n$ and $S \subset G\setminus \{0_G\}$ is such that $|S| \leq c\sqrt{|G|/\log |G|}$. Then there is a set $A \subset G$ such that $S=(A+A)^c$.
\end{theorem}
Here, of course, the set $A$ produced is rather thin and certainly nowhere near the densities we are looking for.

One may also reasonably ask what happens in the transition to the integers.  Theorem \ref{thm.new2} can be proved there through the machinery of Bohr sets as developed by Bourgain \cite{JB}.  However, as the reader will have realised from the proof, the strength comes from the rather precise subgroup structure which is not present in general and as a result the conclusions are weaker than what is already known \cite{BJGAA}.

\section*{Acknowledgements}

The author should like to thank Endre Szemer{\'e}di and Julia Wolf for useful conversations.

\bibliographystyle{alpha}

\bibliography{master}

\end{document}